%% file: genericorthotopes_arxiv.tex
\documentclass{article}
\usepackage{amsmath,amsgen,amstext,amsbsy,amssymb,amsopn,amsfonts,amsthm}
\usepackage{graphics,graphicx,color,lscape,subfigure,chngcntr}
\usepackage{authblk,scalerel}

\graphicspath{{graphics/}}
\counterwithin{figure}{section}

\newtheorem{Def}{Definition}[section]
\newtheorem{Thm}[Def]{Theorem}
\newtheorem{Prp}[Def]{Proposition}
\newtheorem{Lem}[Def]{Lemma}

\def\ar{\scalerel*{\includegraphics{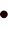}}{]}}
\def\arx{\scalerel*{\includegraphics{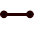}}{]}}

\def\arxxxx{\scalerel*{\includegraphics{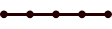}}{]}}
\def\arxxxxn{\scalerel*{\includegraphics{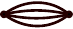}}{]}}
\def\arxxxnx{\scalerel*{\includegraphics{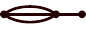}}{]}}
\def\arxxxnxn{\scalerel*{\includegraphics{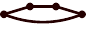}}{]}}
\def\arxxnxx{\scalerel*{\includegraphics{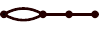}}{]}}
\def\arxxnxxn{\scalerel*{\includegraphics{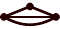}}{]}}
\def\arxxnxnx{\scalerel*{\includegraphics{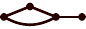}}{]}}
\def\arxxnxnxn{\scalerel*{\includegraphics{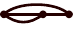}}{]}}
\def\arxxuxx{\scalerel*{\includegraphics{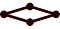}}{]}}
\def\arxxuxxn{\scalerel*{\includegraphics{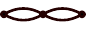}}{]}}

\begin{document}

\title{Generic Orthotopes}

\author[$\star$]{David Richter}

\affil[$\star$]{Department of Mathematics, Western Michigan University, Kalamazoo MI 49008 (USA),
\texttt{david.richter@wmich.edu}}

\maketitle

\begin{abstract}
This article studies a large, general class of orthogonal polytopes
which we may call {\it generic orthotopes}.  
These objects emerged from a desire to represent a Coxeter complex by an
orthogonal polytope that is particularly
nice with respect to traditional topological, structural, or combinatorial considerations.
Generic orthotopes have a pleasant ``homogeneity'' property, somewhat like a
smoothly bounded compact subset of Euclidean space.  Thus, as soon as we demand that
every vertex of an orthogonal polytope be a floral arrangement, as defined here,
many derivative structures such as faces and cross-sections are also described by floral arrangements.
We also give formulas for the volume and Euler characteristic of a generic
orthotope using a couple of statistics that are defined naturally for floral arrangements.
\end{abstract}

\section{Introduction}

Suppose $d$ is a non-negative integer.  By an {\it orthogonal polytope} we mean a union of finitely many
axis-aligned boxes in Euclidean space $\mathbb{R}^d$.  
This article lays a foundation for a theory of 
a particular set of orthogonal polytopes which represents an elementary generalization of the $d$-dimensional
cube to $d$-dimensional orthogonal polytopes.  
We summarize the  salient properties of these ``generic orthotopes'':

\begin{itemize}

\item
Every face of a generic orthotope is a generic orthotope.

\item
Every orthographic cross-section of a generic orthotope is a generic orthotope.

\item
The vertex figure of every vertex a generic orthotope is a simplex.

\item
The 1-dimensional skeleton of a generic orthotope is a bipartite $d$-regular graph.

\item
There are elementary formulas which relate the volume and Euler characteristic
of a generic orthotope.

\item
The structural and combinatorial properties of a generic orthotope remain intact through small perturbations of their facets.

\item
We may approximate any compact subset of Euclidean space to any degree of accuracy with a generic orthotope.

\end{itemize}

Notice that all but the last of these properties 
remain valid when ``generic orthotope'' is replaced by
the word ``cube''.  We establish all of these properties in this article. 
Moreover, by what seems like good fortune, all of these properties follow in an elementary
manner, given an understanding of the local structure of a generic orthotope.

The local structure of a generic orthotope has
a convenient construction using read-once Boolean functions.
Thus, one finds ``floral arrangements'' and ``floral vertices''
at the core of this theory, where
a floral arrangement is determined by applying a read-once Boolean function to a set
of half-spaces possessing distinct supporting hyperplanes.  
We may encode a read-once Boolean function by a series parallel diagram, and
this leads to another bit of good fortune:  We may use a topological invariant of these
diagrams, namely the number of loops modulo 2, to obtain an expression
for the Euler characteristic of a generic orthotope using only the values of this invariant
at its vertices.
Our statement of this formula appears below as Theorem \ref{eulercharthm}.
If one accepts the thesis that generic orthotopes are analogous to
smoothly-bounded subsets of Euclidean space,
then one cannot help but recognize the similarity
of this formula to the Poincar\'e-Hopf theorem which expresses the Euler characteristic
as the sum of indices of a vector field with a finite number of singularities.

The emergence of generic orthotopes is somewhat convoluted.  The original
motivation came from this author's desire to represent 
Coxeter complexes by orthogonal polytopes
which are somehow ``nice''.  In conceiving this problem, however, 
it was not clear what ``nice'' should mean with regards to orthogonal polytopes.  
This author regards convex polytopes which are {\it simple}
(having exactly $d$ edges at every vertex) as particularly ``nice'', but
it was not a priori clear what higher structure one might borrow to study orthogonal polytopes.
The present article precisely develops what ``nice'' should mean for an
orthogonal polytope, and we pose the general problem for
Coxeter complexes in the concluding section.

\subsection{Examples in Low Dimensions}

In order to illustrate the main ideas of this article, we consider the cases $d=2$ and $d=3$.

In two dimensions, we draw a contrast between the polygons which appear
in Figure \ref{orthogons}; one of these is homeomorphic to a disc, while the
other has ``singular points'' where the boundary is self-intersecting.
Using the terminology developed here, the former of these is a generic orthotope
and the latter is not.  
One may relate the numbers of corners of the two types that one sees in a
generic orthogon.  In the polygon on the left in Figure \ref{orthogons}, one
notices that there are $n_1=9$ corners that ``point outward'' and $n_3=5$ corners
that ``point inward''.  The authors of \cite{DV_2005,BDPR_2007} call these ``salient'' and
``reentrant'' points, respectively.  The subscripts 1 and 3 here specify the number
of quadrants occupied by the polygon at that type of vertex.  An immediate corollary
of the 2-dimensional version of our formula in Theorem \ref{eulercharthm} is that
one always has $n_1-n_3=4$ for every generic orthogon.

\begin{figure}
    \centering
    \subfigure[]{\includegraphics[width=0.25\textwidth]{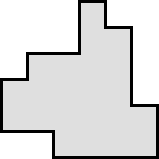}} 
    \hspace{0.05\textwidth}
    \subfigure[]{\includegraphics[width=0.25\textwidth]{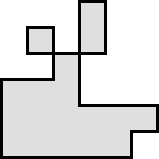}} 
    \caption{(a) A generic orthogon.  (b) Self-intersecting boundary.}
    \label{orthogons}
\end{figure}

As one would expect, the situation when $d=3$ is more complicated.
In the terminology developed here, each of the vertices which appear
in Figure \ref{good3dvertices} is a floral vertex.  Superficially,
the properties that make these vertices ``nice'' are 
(a) there are easily identified faces incident to the vertex and (b) the faces
incident to each vertex coincide with the face lattice of a 2-dimensional simplex
(i.e.\ a triangle).  By contrast, if a 3-dimensional orthogonal polytope has
a ``degenerate vertex'' as one appearing in Figure \ref{false3dvertices}, then we
do not regard it as a generic orthotope.  One can quickly conceive of
other kinds of degenerate points in 3 dimensions, and one imagines that the
number of types of degeneracies that might arise when $d\geq 4$ grows quickly, perhaps exponentially.
Up to congruence, the only types of floral vertices when $d=3$ appear in Figure \ref{good3dvertices}.

\begin{figure} 
\centering 
\includegraphics[width=0.75\textwidth]{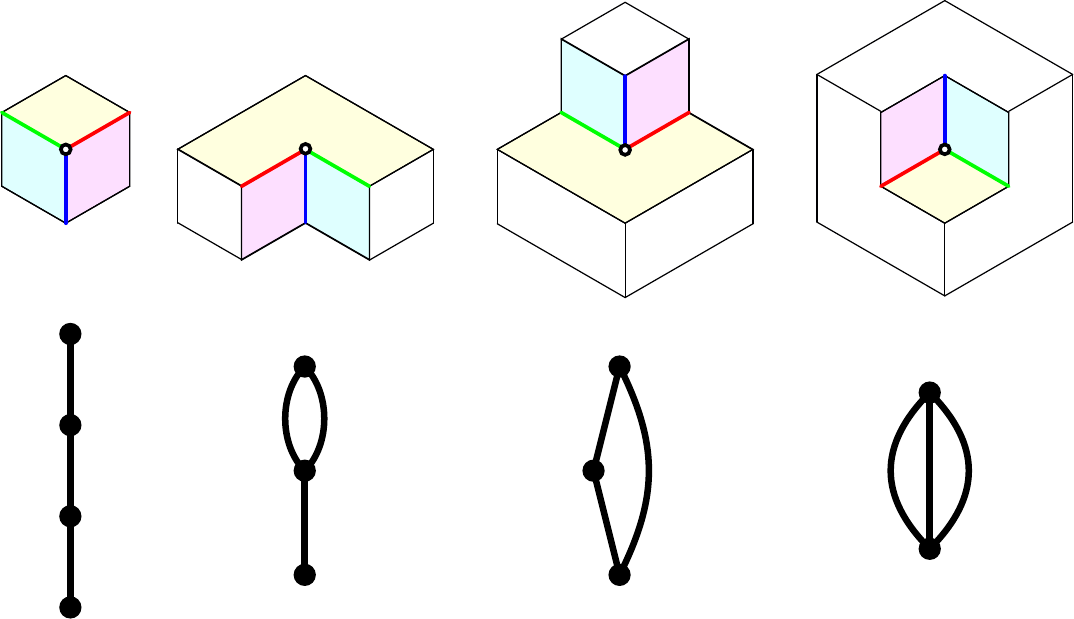}
\caption{Three-dimensional floral vertices.}
\label{good3dvertices}
\end{figure}

\begin{figure} 
\centering 
\includegraphics[width=0.5\textwidth]{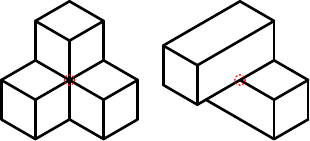}
\caption{Degenerate vertices in 3 dimensions.}
\label{false3dvertices}
\end{figure}

\begin{figure}
$$S=\left\{\begin{array}{ccccccc}
(0,0,0), & (1,0,0), & (2,0,0), & (3,0,0), & (4,0,0), & (0,1,0), & (1,1,0), \\
(2,1,0), & (3,1,0), & (0,2,0), & (2,2,0), & (0,3,0), & (1,3,0), & (2,3,0), \\
(1,0,1), & (2,0,1), & (3,0,1), & (4,0,1), & (2,1,1), & (3,1,1), & (2,2,1), \\
(0,3,1), & (1,3,1), & (2,3,1), & (1,0,2), & (2,0,2), & (3,0,2), & (4,0,2) \\
\end{array}\right\}.$$
\caption{Generating corners of the example.}
\label{examplecorners}
\end{figure}

We may relate the numbers of congruence types of vertices which appear in such
a polytope.  Thus, suppose $P$ is a 3-dimensional orthogonal polytope 
that such that every vertex appears as one of the four congruence types
as depicted in Figure \ref{good3dvertices}.  
For each $i\in\{1,3,5,7\}$, let $n_i$
denote the number of vertices of these corresponding types, where $i$ indicates the
number of octants occupied by its tangent cone.  Then Theorem \ref{eulercharthm} yields,
for $d=3$,
$$n_1-n_3-n_5+n_7=8\chi(P),$$
where $\chi(P)$ is the (combinatorial) Euler characteristic of $P$.

We illustrate some of these ideas with an example.
Define an orthogonal polytope by $P=\bigcup_{v\in S}\left(v+[0,1]^3\right)$,
where $S\subset\mathbb{R}^3$ appears in Figure \ref{examplecorners},
and $v+[0,1]^3$ denotes the translation of the unit cube $[0,1]^3$ by adding $v$.
One should imagine $P$ as an assembly of several 
stacks of unit cubes resting ``skyscraper style'' on a flat surface representing the $(x,y)$-plane
in $\mathbb{R}^3$.  A view of $P$ ``from above'' appears in Figure \ref{exampleehrhart}.  
The numbers which appear in the figure give the heights of these stacks. 
In order to see that $P$ is a generic orthotope, notice
that every vertex of $P$ is congruent to one vertices appearing
in Figure \ref{good3dvertices}.   Here, we have $n_1=15$, $n_3=11$, $n_5=5$, $n_7=1$, and thus
$$n_1-n_3-n_5+n_7=15-11-5+1=0,$$
which we expect because $P$ is homeomorphic to a solid 3-dimensional torus.

\begin{figure} 
\centering 
\includegraphics[width=0.333\textwidth]{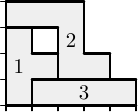}
\caption{A 3-dimensional generic orthotope.}
\label{exampleehrhart}
\end{figure}

\subsection{Flowers}

Dating to the 1950's, the Kneser-Poulsen conjecture asserts roughly that the union of a finite set of Euclidean balls
cannot increase if the distances between their centers decrease or remain equal.
In their work on this problem for general $d$, Bezdek and Connelly \cite{BC_2008} demonstrated
equivalences between various statements which generalize the conjecture to flower
weight functions.  In their terminology, which can be traced to work by 
Csik\'{o}s \cite{csikos_2001} and earlier to Gordon
and Meyer \cite{GM_1995}, a ``flower'' is a subset of Euclidean space that is obtained
by applying a certain type of Boolean function to a collection of balls.
In their description, Bezdek and Connelly used the phrase ``exactly once'' to describe
the variables used for the Boolean functions which they employed to define flowers.
This led the present author to \cite{CH_2011}, where such functions are described as 
``read-once Boolean functions'', terminology which this author employs throughout.
Thus, the underlying construction of a floral arrangement as defined here
is essentially identical with the construction of flowers.

We can see how to use read-once functions to define 3-dimensional floral vertices.
Thus, denote $H_1=\{(x,y,z):x\geq 0\}$, 
$H_2=\{(x,y,z):y\geq 0\}$, and $H_3=\{(x,y,z):z\geq 0\}$ as three 
half-spaces in $\mathbb{R}^3$.   Using union $\cup$ and $\cap$ for union and intersection
respectively, we may describe these four configurations by
$$H_1\cap H_2\cap H_3,\ 
(H_1\cup H_2)\cap H_3,\
(H_1\cap H_2)\cup H_3,\hbox{ and }
H_1\cup H_2\cup H_3.$$ 
By contrast, neither of the degenerate vertices depicted in Figure \ref{false3dvertices}
possesses such a representation.
The graphs which appear in Figure \ref{good3dvertices} are the representations of these
read-once functions by series-parallel diagrams, where $\cap$ is interpreted as series connection and
$\cup$ is interpreted as parallel connection.  In the formula $n_1-n_3-n_5+n_7=8\chi(P)$ for a
3-dimensional generic orthotope, the signs appearing as coefficients on $n_i$ are determined simply by
counting the number of loops in the corresponding diagrams modulo 2.

There is a substantial body of literature on read-once functions (with \cite{CH_2011} as a good
starting point), but most of this
is not relevant for the present work.  A reader of this work should acquaint themself with
the basic notions of read-once functions, as these are fundamental for the definitions
of floral arrangements and generic orthotopes.  
By a similar token, this article does not address the Kneser-Poulsen conjecture.
The works cited above are relevant only insofar as these led this author to incorporate
read-once functions into the theory of orthogonal polytopes.
 
\subsection{Contexts}

The audience for this work includes anyone who has interest in orthogonal
polytopes generally.
This author is impressed by the intrinsic
beauty of generic orthotopes and believes they are worthy of study for their own sake.
The contexts for orthogonal polytopes are certainly myriad, although 
we observe that there is a considerable gap in the theory
of general $d$-dimensional orthogonal polytopes.  Thus, whereas
many workers have studied
rectilinear polygons, polyominoes, polycubes, rectangular layouts, 
orthogonal graph drawings, orthogonal polyhedra/surfaces, $xyz$ polyhedra,
3D staircase diagrams, and so on, comparatively few have focused on the general case when
$d\geq 4$.

With that said, the literature on general $d$-dimensional orthogonal polytopes is 
not completely barren.
Breen has a series of studies, starting with \cite{breen_1996},
which purport to seek analogues of Helly's theorem on convex
sets in Euclidean space for orthogonal polytopes, extending work of
of Danzer and Gr\"unbaum \cite{DG_1982}.  
Starting approximately with \cite{AB_2006} and \cite{BBR_2010}, Barequet and several of his colleagues
have worked on enumerating polycubes (also known as lattice animals) 
for general $d$;
see the aforementioned articles and \cite{guttmann_2009} for more details, especially
as these problems arise in statistical physics.
Werman and Wright \cite{WW_2016} study probabilistic aspects of random cubical complexes;
their approach parallels and complements the present work as they also use the language of valuations for
subsets of $\mathbb{R}^d$.  

Bournez, Maler, and Pnueli \cite{BMP_1999}, concerned with devising ``hybrid systems'' in control theory and
recognizing a gap in the general theory, develop
algorithms for membership, face-detection and Boolean operations for representing these
systems by orthogonal polytopes.
Quoting from \cite{BMP_1999}, 
``Beyond the original motivation coming from computer-aided control system design, we 
believe that orthogonal polyhedra and subsets of the integer grid are fundamental 
objects whose computational aspects deserve a thorough investigation.''
This author believes that the present work advances this project significantly.
Like the present work, the authors of \cite{BMP_1999} use ideas from Boolean algebra, although
they do not use read-once functions to describe local structure of an orthogonal polytope.

P\'{e}rez-Aguila and his colleagues have devoted significant energies
to studying general orthogonal polytopes from the perspective of computer science and computer
engineering, \cite{PA_2006,PA_2010,PA_2011,PAAR_2008}.  
Their approach is largely founded on P\'{e}rez-Aguila's $d$-dimensional generalization
of the Extreme Vertices Model for 3-dimensional orthogonal polytopes
(cf. \cite{AA_2001}).  This model appears 
closely related to the treatment of floral vertices shown here.  Moreover, their
perspective also shares similar significant structural and combinatorial considerations of
orthogonal polytopes with this author.  These works also employ Boolean algebra extensively,
although again we notice a lack of emphasis on read-once Boolean functions in particular.

Orthogonal polytopes arise in toric geometry, where they are called ``staircases''.
A basic idea in this theory is that we can gain some algebraic insight by modeling
a square-free monomial ideal in a polynomial ring by studying an associated orthogonal polytope
lying in the primary orthant (where all coordinates are non-negative) 
of $\mathbb{R}^d$.  The fascinating text \cite{MS_2005} expounds on these ideas at great length.
However, we stress that the orthogonal polytopes most often encountered in toric geometry appear to be ``totally
spherical'' in the sense that every face is homeomorphic to a closed cell, whereas the
objects studied here are considerably more general.

A generic orthotope shares some attributes with a smoothly bounded compact Euclidean set.
For example, the tangent cone at every point on the boundary of any given generic orthotope is homeomorphic
to a half-space.  Similarly, the bipartiteness of the 1-dimensional skeleton of a generic orthotope
is reminiscent of the orientability the boundary of a smoothly bounded compact set.
Moreover, our function $\sigma$ 
seems to measure a discrete analogue of curvature at each point.
In the same vein, we also mention this author's recent work
\cite{richter_rectangulations} on generic rectangulations.  In its roughest description, a generic rectangulation
is a subdivision of a rectangle of a rectangle into rectangles, where the descriptor ``generic''
means that no four constituent rectangles share a common corner.
In \cite{richter_rectangulations},
this author demonstrated that one may perform a ``central involution'' on a generic rectangulation
about any one of its consituent rectangles, analogous
to linear fractional transformations of the complex projective line $\mathbb{C}P^1$.
Since generic orthotopes are defined discretely, this indicates a context 
in discrete differential geometry.  

\subsection{Organization}

This article is organized as follows.  First there is a brief section on background on orthogonal
polytopes; this defines tangent cones and the face poset of an orthogonal polytope.
Following this are two long sections which describe the foundations of generic orthotopes.
These are divided according to local theory versus global theory.  
The section on local theory defines floral arrangements, studies
the structure of a floral arrangement, and introduces two ``local'' valuations $\mu$ and $\tau$ which
will be required in the section on global theory.  The section on global theory defines
generic orthotopes (in terms of floral arrangements), studies some of their properties,
and gives several formulas relating the volume and Euler characteristic of a generic orthotope.
We conclude with a few open questions about generic orthotopes.

Throughout this article, denote $\left[d\right]=\{1,2,3,...,d\}$.

\section{General orthogonal polytopes}

Denote the standard basis for $\mathbb{R}^d$ by $\{e_i:i\in\left[d\right]\}$, where $e_i$
is the unit vector pointing along the positive $x_i$-axis for each $i$.
The {\it cardinal directions} are the $2d$ vectors $\{\pm e_i:i\in\left[d\right]\}$.
The {\it cardinal ray of $\delta$} is the cone generated by the cardinal direction $\delta$.

For each $i\in\left[d\right]$ and for each $\lambda\in\mathbb{R}$,
let $\Pi_{i,\lambda}$ be the $(d-1)$-dimensional hyperplane defined by the equation $x_i=\lambda$,
(i.e.\ the null-space of the functional $x_i-\lambda$).  
Define the {\it $i$th canonical orthographic projection} by
$$\pi_i:(x_1,x_2,...,x_i,...,x_d)\mapsto(x_1,x_2,...,0,...,x_d).$$
We identify each hyperplane $\Pi_{i,\lambda}$ with $\mathbb{R}^{d-1}$ via the $i$th orthographic projection.
If $I\subset\left[d\right]$ is any subset and $\lambda:I\rightarrow\mathbb{R}$ is a tuple, let 
$$\Pi_{I,\lambda}=\bigcap_{i\in I} \Pi_{i,\lambda(i)}$$
be the {\it generalized hyperplane determined by $(I,\lambda)$}.

An {\it axis-aligned box}
is a cartesian product of closed intervals
$\prod_{i=1}^d\left[a_i,b_i\right]$
such that $a_i\leq b_i$ for all $i$ and
each interval $\left[a_i,b_i\right]$ is embedded in the $i$th summand of the
direct sum $\mathbb{R}^d=\bigoplus_{i=1}^d\mathbb{R}$.
Call such a box {\it pure $d$-dimensional} if $a_i<b_i$ for all $i$.  
An {\it orthogonal polytope} is a subset of $\mathbb{R}^d$ that has an expression as 
the union of a finite set of axis-aligned boxes.
Call an orthogonal polytope {\it pure $d$-dimensional} if it has an expression as a union of pure
$d$-dimensional axis-aligned boxes and {\it singular} if it admits no such expression.

Axis-aligned boxes are the fundamental examples of orthogonal polytopes.
(In fact, the term ``orthotope'' has been used for these objects, cf.\ \cite{coxeter_1963}.)
The {\it standard unit cube} is the cartesian
product $I^d$, where $I=\left[0,1\right]$ is the unit interval.
An orthogonal polytope $P$ is {\it integral} if it is a union of translates
$$P=\bigcup_{v\in S} (v+f_v)\hbox{ (Minkowski sum)},$$
where $S$ is a finite subset of the lattice $\mathbb{Z}^d\subset\mathbb{R}^d$
and $f_v$ is a face of the standard unit cube $I^d$ for all $v\in S$.
Call an orthogonal polytope $P$ {\it rational} if there is a positive integer $n$ such that
$nP$ is integral.

\subsection{Tangent cones and faces}

We aim here to define the face poset of an orthogonal polytope.

Given $s=(s_1,s_2,...,s_d)\in\{\pm 1\}^d$, the {\it orthant represented by $s$} is
$\Omega_s=\{(x_1,x_2,...,x_d):s_ix_i\geq 0\hbox{ for all }i\}.$
A {\it local orthotopal arrangement in $\mathbb{R}^d$} is a union of orthants.
Local orthotopal arrangements appear in bijective correspondence with Boolean functions.
Define the {\it sign} of $s$ as $(-1)^s=\prod_{i=1}^d s_i$.

Suppose $P\subset\mathbb{R}$ is a pure $d$-dimensional orthotope and $v\in P$.
The {\it tangent cone at $v$} is the local orthotopal arrangement $\alpha$ such that
there exists $\delta>0$ such that
$$P\cap [-\delta,\delta]^d=v+\left(\alpha\cap [-\delta,\delta]^d\right).$$

If $v\in P$, define the {\it genericity region} of $v$ as the set of all points $w$ that can
be joined by a path $\gamma$ for which the tangent cone at every point along
$\gamma$ is congruent to the tangent cone at $v$.  
Evidently every genericity region is path-connected and the genericity
regions partition $P$.  The {\it degree} of a genericity region is the smallest dimension
among the hyperplanes which contain it.
Thus, if $v$ lies interior to $P$ if $v$ has genericity degree $d$, and $v$ is a {\it singular point}
of $P$ if its degree of genericity is zero.
Define a {\it $k$-dimensional face} of $P$ as the closure of a $k$-dimensional genericity
region.
The {\it face poset} of $P$ is the set of all faces of $P$, partially ordered by inclusion.

Suppose $\alpha\subset\mathbb{R}^d$ is a local orthotopal arrangement.  
Define $\mu_d(\alpha)$ as the number of orthants occupied by $\alpha$
and let $\tau_d(\alpha)$ denote the sum of the signs
of the orthants occupied by $\alpha$.  Apparently 
these functions satisfy inclusion-exclusion identity
$$f(\alpha)+f(\beta)=f(\alpha\cap\beta)+f(\alpha\cup\beta),$$
for $f\in\{\mu_d,\tau_d\}$ and local orthotopal arrangements $\alpha,\beta$.

\section{Generic Orthotopes: Local Theory}

The purpose here is to introduce and study floral arrangements,
which will be needed in order to define generic orthotopes.  
First we explain how to use series-parallel diagrams to define floral
arrangements and floral vertices.
Then we describe some structure of the face lattice of a floral arrangement.
Finally we introduce some ``local'' valuations and relate them to the functions
$\mu_d$ and $\tau_d$ defined above.

\subsection{Series-parallel diagrams}

Define a {\it series-parallel diagram} (SPD for short) inductively as either

\begin{enumerate}

\item
A single edge $\arx$ joining two terminals (the vertices),

\item
a series connection of series-parallel diagrams, or

\item
a parallel connection of series-parallel diagrams.

\end{enumerate}

We admit the single vertex $\ar$ as an ``honorary'' SPD, even though it does not
possess two distinct terminals.

\subsubsection{Bouquet sign}

If $\Delta$ is an SPD, let $E(\Delta)$ denote the edges of $\Delta$, and let
$e(\Delta)=|E(\Delta)|$ and $v(\Delta)$ denote the numbers of edges and vertices of $\Delta$ respectively.
Use the symbols $\wedge$ and $\vee$ to denote series and parallel connection, respectively.
Then we have
$$v(\arx)=2,\ e(\arx)=1,$$
$$e(\Delta_1\vee\Delta_2)=e(\Delta_1\wedge\Delta_2)=e(\Delta_1)+e(\Delta_2),$$
$$v(\Delta_1\wedge\Delta_2)=v(\Delta_1)+v(\Delta_2)-1,$$
and
$$v(\Delta_1\vee\Delta_2)=v(\Delta_1)+v(\Delta_2)-2$$
for any SPD's $\Delta_1$ and $\Delta_2$.

Define the {\it bouquet rank} of an SPD $\Delta$ as
$$\rho(\Delta):=e(\Delta)-v(\Delta)+1.$$
The terminology comes from the fact that an SPD $\Delta$ is homotopy equivalent to a bouquet
of $\rho(\Delta)$ circles.
Evidently we have
$$\rho(\Delta_1\wedge\Delta_2)=\rho(\Delta_1)+\rho(\Delta_2)\hbox{ and }
\rho(\Delta_1\vee\Delta_2)=\rho(\Delta_1)+\rho(\Delta_2)+1$$
for any $\Delta_1$, $\Delta_2$.
Define the {\it bouquet sign} of $\Delta$ by
$$\sigma(\Delta):=(-1)^{\rho(\Delta)}.$$
Evidently we have
$$\sigma(\Delta_1\wedge\Delta_2)=\sigma(\Delta_1)\sigma(\Delta_2)
\hbox{ and }
\sigma(\Delta_1\vee\Delta_2)=-\sigma(\Delta_1)\sigma(\Delta_2).$$
for all $\Delta_1,\Delta_2$.

\subsubsection{Duality}

Suppose $\Delta$ is an SPD.  The {\it dual} of $\Delta$
is the SPD $\overline{\Delta}$ obtained by interchanging the roles of series
and parallel connection in its parse tree.
If $\Delta$ has $d$ edges, then the dual $\overline{\Delta}$ also has $d$ edges,
and the bouquet rank $\rho$ satisfies
$\rho(\Delta)+\rho(\overline{\Delta})=d-1$ for all $\Delta$.
Accordingly, the bouquet sign $\sigma$ satisfies
$\sigma(\Delta)\sigma(\overline{\Delta})=(-1)^{d-1}$ for every $\Delta$
with $d$ edges.

A {\it signed SPD} is a pair $(\Delta,s)$, where $\Delta$ is an SPD 
and $s:E(\Delta)\rightarrow\{\pm 1\}$.
In drawing signed SPD's by hand, it is convenient to
indicate negative edges with overline bars and positive edges without such marks
or by using distinguishing colors.
Define the {\it dual} of $(\Delta,s)$
as the signed SPD $(\overline{\Delta},-s)$. 
Notice one obtains the dual $(\overline{\Delta},-s)$ by
applying DeMorgan's laws when we interpret $(\Delta,s)$ as a Boolean function.
Figure \ref{dualdiagram} displays an example of a signed series-parallel 
diagram and its dual.

\begin{figure}
\centering
\def\svgwidth{0.5\textwidth}
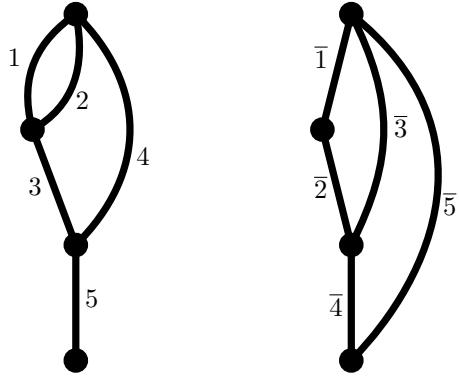
\caption{A signed series-parallel diagram and its dual.}
\label{dualdiagram}
\end{figure}

\subsection{Floral arrangements and floral vertices}

Suppose $(\Delta,s)$ is a signed SPD with edges $E\subset\left[d\right]$, and let
$\{H_i:i\in E\}$ be the positive closed half-spaces in $\mathbb{R}^d$ such
that $H_i$ is supported by the hyperplane $\Pi_{i,0}$ for each $i$.
The {\it floral arrangement} determined by $(\Delta,s)$ is the local orthotopal
arrangement obtained by interpreting series
connection as intersection, parallel connection as union,  
and each $i\in E(\Delta)$ is evaluated as $s_iH_i$.
We use the term {\it floral vertex} in the case when $E=\left[d\right]$. 
Figure \ref{flowers2d} illustrates this for $d=2$.
It is apparent that every floral arrangement $\alpha=\alpha(\Delta,s)$ has an expression
$\alpha\cong\mathbb{R}^{d-|E|}\times\alpha'$,
where $\alpha'\subset\mathbb{R}^{|E|}$ is a floral vertex on $|E|$ half-spaces.
If $\alpha$ is a floral arrangment defined by the signed SPD 
$(\Delta,s)$, let $\overline{\alpha}$ denote the complementary arrangement
defined by the dual $(\overline{\Delta},-s)$.

At this point, it is important to note a particular usage of the symbols
$\wedge$, $\vee$, $\cap$, and $\cup$.
If $\alpha,\beta\subset\mathbb{R}^d$ are floral arrangements, then 
$\alpha\cap\beta,\alpha\cup\beta\subset\mathbb{R}^d$ are interpreted using ordinary 
intersection and union.  However, we cannot expect $\alpha\cap\beta$
or $\alpha\cup\beta$ to be a floral arrangement in general.
For example, if $\alpha$ is represented by $(1\vee 2)\wedge 3$
and $\beta$ is represented by $(1\vee 3)\wedge 2$, then 
$\alpha\cup\beta$ is the degenerate arrangement with 6 edges as depicted
in Figure \ref{false3dvertices}.
On the other hand, if $\alpha$ and $\beta$
are floral vertices, then we use $\alpha\wedge\beta$ or $\alpha\vee\beta$
to denote the floral vertex obtained by joining the SPD's for $\alpha$
and $\beta$ in series or parallel, respectively.  In particular, if $\alpha$
and $\beta$ are floral vertices, then $\alpha\wedge\beta$ is a cartesian product
of $\alpha$ and $\beta$.  This is evident, for example, in Figure \ref{flowers2d}.

\begin{figure}
\centering
\def\svgwidth{0.75\textwidth}
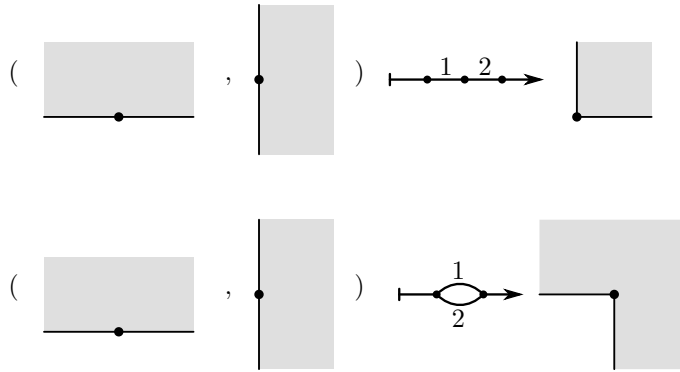
\caption{Floral vertices in $\mathbb{R}^2$.}
\label{flowers2d}
\end{figure}

We consider congruence types of floral arrangements.
The group of symmetries of the cube $[-1,1]^d$ acts on 
acts on floral arrangements $\alpha\subset\mathbb{R}^d$ in an apparent way.
Denote this group by $BC_d$, and recall that
$BC_d$ is isomorphic to the wreath product $S_2\wr S_d=S_d\ltimes S_2^d$, where
$S_d$ acts on $S_2^d$ by permuting the coordinates and $S_2^d$
acts by altering the signs of the coordinates.  Every element
$g\in BC_d$ may thus be regarded as an ordered pair $g=(f,s)$, where $s\in\{\pm 1\}^d$
is a tuple of signs and $f$ is a permutation of $\left[d\right]$.  Suppose $g=(f,s)\in BC_d$
and $(\Delta,s')$ is a signed SPD with edge set $I\subset\left[d\right]$.  Then the action
of $(f,s)$ on $(\Delta,s')$ is the signed SPD with edge set $f(I)$ and sign function $ss'$ obtained
by coordinate-wise multiplication.  
We may represent each orbit in this action by an SPD where we do not distinguish the edges.
We refer to the {\it dimension} of a floral vertex $\alpha$ as the number $\dim\alpha=e(\Delta)$ of
edges used in the diagram $\Delta$ which defines $\alpha$.

Let $A_d$ denote the number of congruence classes of floral vertices on $d$ edges.
This coincides with the number of unmarked SPD's and appears as sequence A000084 in the Online
Encyclopedia of Integer Sequences.  Thus, the number of congruence classes of floral arrangements
in $\mathbb{R}^d$
is $\sum_{k=0}^d A_k$.  Several values of these sequences appear in the table
in Figure \ref{oeis}.

\begin{figure}
$$\begin{array}{ccc}
d & A_d & \sum_{k=0}^d A_k \\
\hline
0 & 1 & 1 \\
1 & 1 & 2 \\
2 & 2 & 4 \\
3 & 4 & 8 \\
4 & 10 & 18 \\
5 & 24 & 42 \\
6 & 66 & 108 \\
7 & 180 & 288 \\
8 & 522 & 810 \\
9 & 1532 & 2342 \\
10 & 4624 & 6966 \\
\end{array}$$
\caption{The numbers of floral vertices and floral arrangements.}
\label{oeis}
\end{figure}

\subsection{Facets of a floral vertex}

We demonstrate here that the faces of a floral vertex coincide
with the faces of a simplex and that every face is also described by a floral vertex.
Throughout this section, we assume that  
$\alpha\subset\mathbb{R}^d$ is a floral vertex determined by a signed SPD 
$(\Delta,s)$ with edges $E=\left[d\right]$.  Without loss of generality, we
assume that every component of $s$ is $+1$, so that all of the corresponding half-spaces are positive.
For each $i\in E$, denote 
$\delta_i(\alpha)=(\partial\alpha)\cap\Pi_{i,0}$.

\begin{Prp}
Suppose a floral vertex has an expression $\alpha=\alpha_1\wedge\alpha_2$,
where $\alpha_1$ is a floral vertex on $\{H_1,H_2,...,H_k\}$.
For each $i\in\{1,2,3,...,k\}$, we have 
$\delta_i(\alpha)=\delta_i(\alpha_1)\wedge\alpha_2$.
\end{Prp}

\begin{proof}
This follows from the fact that the floral vertex $\alpha_1\wedge\alpha_2$ is the 
cartesian product of the floral vertices $\alpha_1$ and $\alpha_2$.
\end{proof}

We may now state:

\begin{Prp}
Let $\alpha\subset\mathbb{R}^d$ be a floral vertex, and suppose
$k$ is an integer with $0<k<d$.  Then 
(i) every degree-$k$ genericity region of $\alpha$ is a $k$-dimensional face of $\alpha$, and
(ii) every generalized coordinate hyperplane of dimension $k$ contains precisely one face of $\alpha$.
\end{Prp}

\begin{proof}
This follows by induction.  The statement is apparently valid when $d=1$.
Suppose $d\geq 2$ is fixed and assume the statement is true for all values less than $d$,
and let $\alpha$ be a floral vertex on $d$ half-spaces.
Suppose first that $\alpha$ is a conjunction, say $\alpha=\alpha_1\wedge\alpha_2$.
The induction hypothesis then holds for
$\alpha_1$ and $\alpha_2$.  However, since $\alpha$ is then a cartesian product,
so the conclusion holds.  On the other hand, if $\alpha$ is a disjunction $\alpha_1\vee\alpha_2$, 
then one may
apply this argument to the complementary arrangement $\overline{\alpha}$, which is necessarily
a conjunction.
\end{proof}

As an immediate corollary of the preceding results on the structure of a local
floral arrangement, we see:

\begin{Prp}
Let $\alpha\subset\mathbb{R}^d$ be a floral vertex on the half-spaces
$H_1,H_2,...,H_d$.  Then each facet $\delta_i(\alpha)$ is a floral
vertex on the subspaces $H_j\cap\Pi_{i,0}$, with $j\neq i$.
\end{Prp}

To elucidate this, we combine the preceding results about the
structure of a floral arrangement to describe an
algorithm for computing $\delta_i(\alpha)$.
The input of the algorithm includes a positive integer
$d$, a floral vertex $\alpha$ described by an SPD $\Delta$,
and an integer $i\in\left[d\right]$.  In the first iteration, one checks whether or
not $\Delta$ is a conjunction or a disjunction.  
If $\Delta$ is a conjunction, then we let $\alpha_0=\alpha$.  
If $\Delta$ is a disjunction, then we let $\alpha_0=\overline{\alpha}$.  
Next, assuming that we know $\alpha_k$ from a previous iteration,
we define two new floral arrangements
$\beta_k$ and $\gamma_k$ by the requirement that $\alpha_k=\beta_k\wedge\gamma_k$, where
$\beta_k$ and $\gamma_k$ are floral arrangements such that
$\beta_k$ is a disjunction which contains $i$ or $\overline{i}$.  
Next, we define $\alpha_{k+1}=\overline{\beta_{k}}$.
We iterate this until $\alpha_k$ is a floral arrangement on a single
half-space.  The floral arrangement representing the facet $\delta_i(\alpha)$ is then the
conjunction $\gamma_1\wedge\gamma_2\wedge\cdots\wedge\gamma_k$, where $\gamma_k$
is the last ``difference'' obtained from these iterations.

Figure \ref{facetalgorithm} illustrates an example of this algorithm.
The red numerals indicate negative half-spaces, and the red arrows indicate
complementation.  Here, the input arrangement is
$$\alpha=(((((\mathbf{1}\vee\mathbf{2})\wedge\mathbf{3})\vee\mathbf{4})
\wedge\mathbf{5})\vee\mathbf{6})\wedge(\mathbf{7}\vee\mathbf{8}),$$
represented by the SPD in the figure,
and the output arrangement is
$$\delta_4(\alpha)=(\overline{\mathbf{1}}\vee\overline{\mathbf{2}})
\wedge\mathbf{5}\wedge\overline{\mathbf{6}}\wedge(\mathbf{7}\vee\mathbf{8}).$$
The reader is urged to use this algorithm to compute $\delta_i(\alpha)$ for $i\neq 4$.

\begin{figure} 
\centering 
\includegraphics[width=0.75\textwidth]{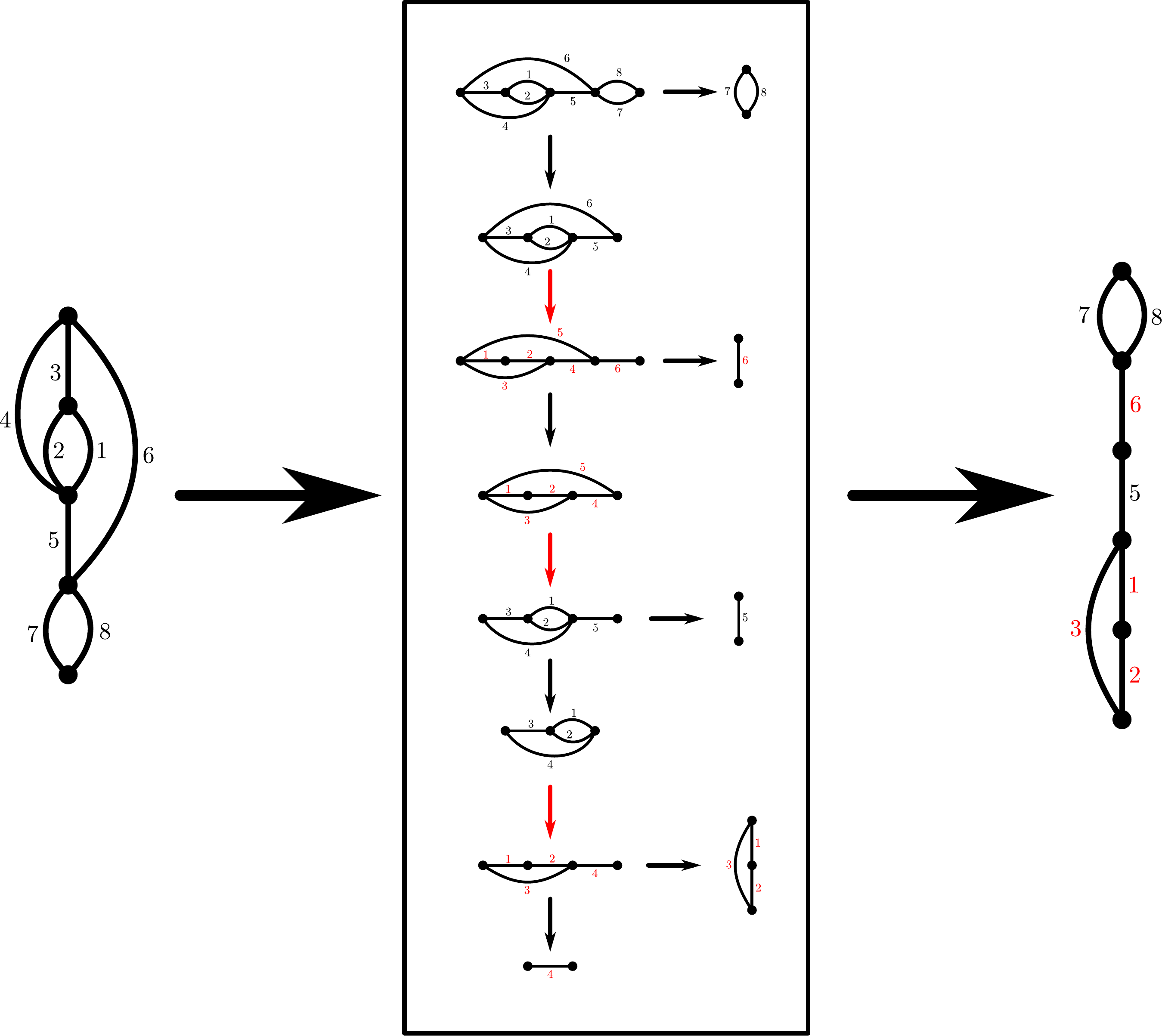}
\caption{Illustration of the facet algorithm, for $\delta_4(\alpha)$.}
\label{facetalgorithm}
\end{figure}

\subsection{Edges and cross-sections}

In this section we explain some structure concerning edges and cross-sections of a floral vertex.  
First we show how to determine the edges emanating from a floral vertex. 
Then we show that each cross-section perpendicular to a given edge is described
by a floral arrangement.

\subsubsection{Edges of a floral vertex}

Let $\alpha$ be a floral vertex.
It is a consequence of the analysis above that for each $i\in\left[d\right]$, 
exactly one of the cones generated by either $e_i$ or $-e_i$ is an edge of $\alpha$.
We demonstrate here how to determine the directions of these edges.

A precise statement of this result is
facilitated with a method of deleting an edge from an SPD, as follows.
Suppose $\Delta$ is an SPD with edges $E=\left[d\right]$.
Choose and fix an edge $i\in E$.  Call $i$ {\it disjunctive} if there is another
subdiagram that is connected in parallel with $i$.
Call $i$ {\it conjunctive} if it is not disjunctive. 
By DeMorgan's laws, an edge $i$ of $\Delta$ is is conjunctive in $\Delta$
if and only if $i$ is disjunctive in the dual $\overline{\Delta}$.

If $i$ is an edge of $\Delta$, define the 
{\it deletion of $i$} as the SPD $\Delta\backslash i$ as follows:

\begin{itemize}

\item
If $i$ is conjunctive, then $\Delta\backslash i$ is the SPD
obtained by deleting $i$ from $\Delta$ and identifying the terminals of $i$.

\item
If $i$ is disjunctive, then $\Delta\backslash i$ is the series parallel diagram
obtained by deleting $i$ from $\Delta$.

\end{itemize}

\begin{Prp}
Suppose $\alpha$ is a floral vertex determined by the signed series-parallel
diagram $(\Delta,s)$, with $s=(s_1,s_2,...,s_d)$ and $i$ is an edge of $\Delta$.  Then
Then the cone generated by $s_ie_i$ is an edge of $\alpha$ if and only if 
$\sigma(\Delta)=\sigma(\Delta\backslash i)$.
\end{Prp}

\begin{proof}
We use an observation about the facet-computing algorithm described above.
Each iteration in the algorithm involves complementation of a subdiagram of
$\Delta$ or $\overline{\Delta}$ which contains $i$ or $\overline{i}$.  Thus, in the last iteration, one
is left with the series-parallel graph consisting of a single edge that is marked
with either $i$ or $\overline{i}$.  Thus, the sign is 
determined by the parity of the number of iterations
required.  
Note that a floral vertex and its complementary arrangement have the same boundary complex,
so, in particular, they have the same edges.
Note also that cartesian products preserve edges in the following sense:
If $e_i$ generates an edge of a floral vertex, say $\alpha$, then $e_i$ also
generates an edge in a conjunction $\alpha\wedge\beta$.
\end{proof}

From this, we see that another way to write this result is that, under
the given hypotheses, the product
$\sigma(\Delta)\sigma(\Delta\backslash i)s_ie_i$ generates an edge for each $i$.

\subsubsection{Cross-sections}

Suppose $\alpha\subset\mathbb{R}^d$ is a floral arrangement.  
For our purposes, a {\it cross-section of $\alpha$} is an intersection of the form $\alpha\cap\Pi_{I,\lambda}$, 
where $\Pi_{I,\lambda}$ is a generalized hyperplane as defined earlier.
We show that every cross-section parallel to a generalized coordinate hyperplane
of a floral arrangement is described by another floral arrangement. 

Assume $\alpha\subset\mathbb{R}^d$ is a floral vertex determined
by the signed SPD $(\Delta,s)$.  Without loss of generality,
assume that every component of $s$ is $+1$.  Suppose $i\in E(\Delta)$ is given. 
The structure of a given cross-section $\alpha\cap\Pi_{i,\lambda}$ depends
on the sign of $\lambda$ and the orientation of the edge of $\alpha$ corresponding to $i$.
To distinguish the cases, define
an {\it edge cross-section} of $\alpha$ as a cross-section 
that passes through the relative interior of the edge of $\alpha$ corresponding to $i$, and
define a {\it residual cross-section} of $\alpha$ as a cross section $\alpha\cap(\Pi_{i,\lambda})$
such that $\alpha\cap(\Pi_{i,-\lambda})$ is an edge cross-section.
Following immediately from the definitions, we notice:

\begin{Prp}
Suppose $\alpha$ is a floral vertex represented by the series-parallel
diagram $\Delta$ and $i$ is an edge of $\Delta$.  
The edge cross-section of $\alpha$
corresponding to $i$ is congruent to the floral vertex represented by the series
parallel diagram $\Delta\backslash i$.
\end{Prp}

Next we describe the residual cross-section of $\alpha$ for the edge $i$. 
Given an SPD $\Delta$, define $\Delta_i$ as follows:

\begin{itemize}

\item
If $i$ is conjunctive, then $\Delta_i$ is the maximal subdiagram of $\Delta$ connected in series
with $i$.

\item
If $i$ is disjunctive, then $\Delta_i$ is the maximal subdiagram of $\Delta$ connected in parallel
with $i$.

\end{itemize}

Notice if $i$ is disjunctive, then $\Delta_i$ is the parallel connection of all of the
subdiagrams that share terminals with $i$.  If $i$ is conjunctive, then
one may employ the same idea using duality.  Thus, if $i$ is conjunctive in $\Delta$, then
$i$ is disjunctive in the dual $\overline\Delta$, and one may define the dual of
$\Delta_i$ as the maximal subdiagram of $\overline{\Delta}$ connected in parallel
with $i$, then dualize to obtain $\Delta_i$.  We define the deletion $\Delta\backslash\Delta_i$
analogous to our definition of $\Delta\backslash i$ above.
Now we may state:

\begin{Prp}
Suppose $\alpha$ is a floral vertex determined by the series-parallel
diagram $\Delta$ and suppose $i$ is an edge of $\Delta$.
Then the residual cross-section of $\alpha$ corresponding to $i$ is congruent to
the floral arrangement determined by the series parallel diagram
$\Delta\backslash\Delta_i$.
\end{Prp}

\begin{proof}
By our result above on the orientation of the edge of $\alpha$ corresponding to
$i$, one may obtain the diagram of a residual cross-section 
by conjoining the diagram $\Delta$ with either $\overline{i}$ or $i$ depending
respectively on whether $i$ is conjunctive or disjunctive.  
If $i$ is conjunctive, then conjoining $\overline{i}$ with $\Delta$ kills
all subdiagrams that are connected in series with $i$.
This is a consequence of the set-theoretic fact that
$(\alpha\cap\beta)\cap\overline{\alpha}$ is an empty set for any $\alpha,\beta$.
If $i$ is disjunctive, then conjoining $i$ with $\Delta$ kills all
subdiagrams that are connected in parallel with $i$.
This is a consequence of the set-theoretic fact that
$(\alpha\cup\beta)\cap\alpha=\alpha\cap\beta$ for any $\alpha,\beta$.
The result then follows by induction.
\end{proof}

Figures \ref{conjunctivesection} and \ref{disjunctivesection}
illustrate examples of cross-sections of the four types.  
In both figures, one is given a floral vertex $\alpha\subset\mathbb{R}^8$.
Figure \ref{conjunctivesection} displays the resulting cross-sections for the conjunctive
edge numbered 6, and Figure \ref{disjunctivesection} displays the resulting cross-sections
for the disjunctive edge numbered 1.  In both figures, the edge cross-sections appear on the left
while the residual cross-sections appear on the right.  The reader
is urged to compute edge cross-sections and residual cross-sections for the other
edges.

\begin{figure} 
\centering 
\includegraphics[width=0.75\textwidth]{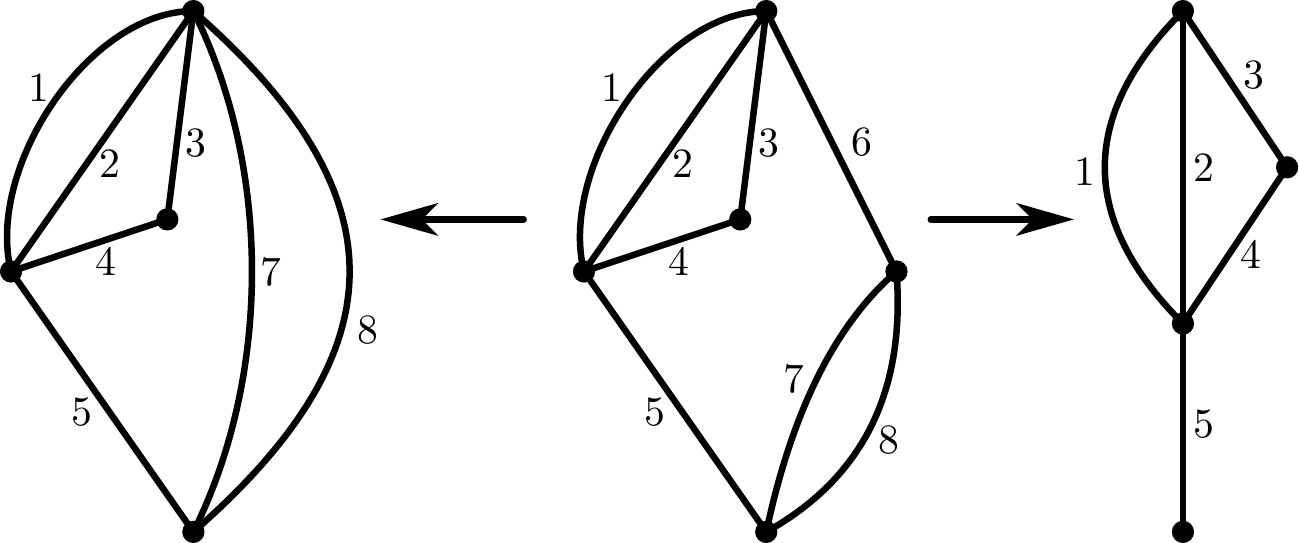}
\caption{Cross-sections for a conjunctive edge.}
\label{conjunctivesection}
\end{figure}

\begin{figure} 
\centering 
\includegraphics[width=0.75\textwidth]{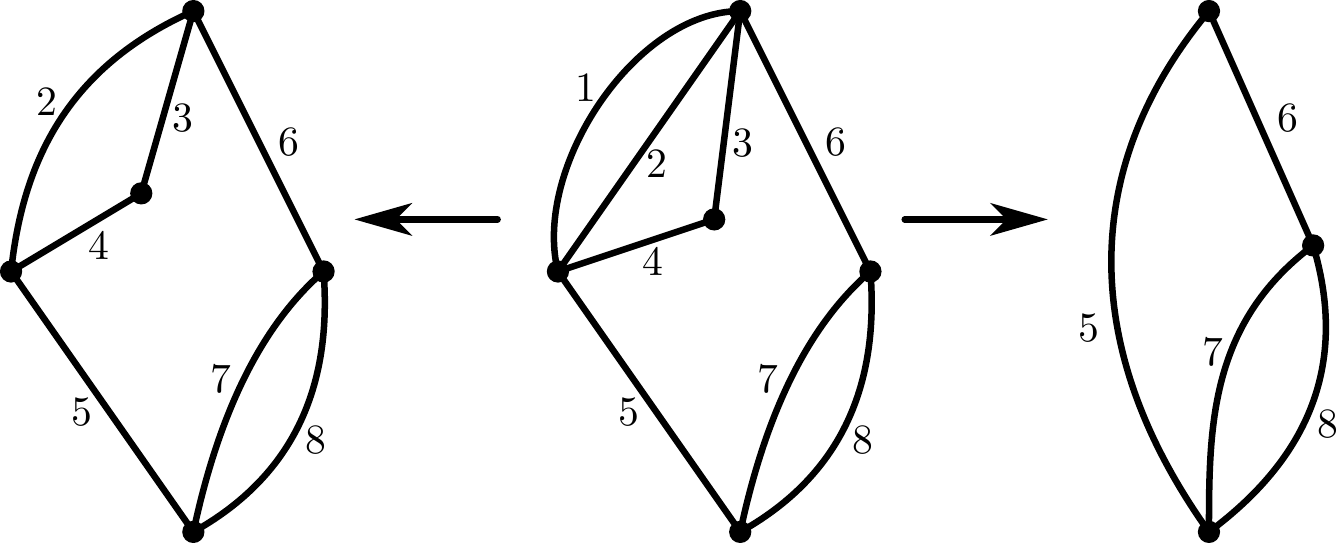}
\caption{Cross-sections for a disjunctive edge.}
\label{disjunctivesection}
\end{figure}

Finally we describe the cross-sections of the form $\alpha\cap\Pi_{i,0}$.
Since we defined a local arrangement as being a union of {\it closed} orthants, 
each such cross-section is congruent
to exactly one of either an edge cross-section or a residual cross-section of $\alpha$.
It can't be both because we assumed that $\alpha$ is a floral vertex.
Thus, in either case, we can describe every cross-section
by a floral arrangement.

\subsection{Local valuations}

In this section we introduce and develop functions $\mu$ and $\tau$ and relate these
to the functions $\mu_d$ and $\tau_d$ introduced above.  These
functions will facilitate the
study of global properties of generic orthotopes below.

\subsubsection{Volume}

First we define $\mu$ on SPDs as follows.  Suppose $\Delta$ is a
SPD.  We define $\mu(\Delta)$ inductively as follows.
\begin{enumerate} 

\item
If $\Delta$ is the SPD with exactly one edge, then $\mu(\Delta)=1$.

\item
$\mu(\Delta_1\wedge\Delta_2)=\mu(\Delta_1)\mu(\Delta_2)$ for all $\Delta_1$, $\Delta_2$.
 
\item
$\mu(\overline{\Delta})+\mu(\Delta)=2^d$, where $\Delta$ has $d$ edges.

\end{enumerate}

This function $\mu$ is evidently related to the functions $\mu_d$:

\begin{Prp}
Suppose $\alpha\subset\mathbb{R}^d$ is a floral arrangement that is given by a
SPD $\Delta$ having $k$ edges.  Then $\mu_d(\alpha)=2^{d-k}\mu(\Delta)$.
\end{Prp}

If $(\Delta,s)$ is a signed SPD on $k$ edges
and $\alpha\subset\mathbb{R}^d$ is a floral arrangement determined by
$(\Delta,s)$, then we abuse notation slightly by defining $\mu(\alpha):=\mu(\Delta)$.

\subsubsection{Signed volume}

Suppose $(\Delta,s)$ is a signed SPD.  We define $\tau(\Delta,s)$ inductively
as follows.

\begin{enumerate}

\item
If $\Delta$ is the diagram with a single edge, then
$\tau(\Delta,1)=1$.

\item
$\tau(\Delta_1\wedge\Delta_2,(s_1,s_2))=\tau(\Delta_1,s_1)\tau(\Delta_2,s_2)$ for all 
$(\Delta_1,s_1)$, $(\Delta_2,s_2)$.
 
\item
$\tau(\overline{\Delta},-s)+\tau(\Delta,s)=0$ for all $(\Delta,s)$.

\end{enumerate}

As above, we can relate $\tau$ and $\tau_d$:

\begin{Prp}
Suppose $\alpha\subset\mathbb{R}^d$ is a floral arrangement determined by a signed
SPD $(\Delta,s)$ having $k$ edges.
on $k$ half-spaces.  Then
$$\tau_d(\alpha)=\left\{\begin{array}{cl}
\tau(\Delta,s) & \hbox{if }k=d, \\
0 & \hbox{if }k<d. \\
\end{array}\right.$$
\end{Prp}

If $(\Delta,s)$ is a signed SPD on $k$ edges
and $\alpha\subset\mathbb{R}^d$ is the floral arrangement determined
$(\Delta,s)$, then we again abuse notation slightly by defining $\tau(\alpha):=\tau(\Delta,s)$.

\begin{Prp}
If $\alpha$ is a floral vertex determined by the signed series-parallel
diagram $(\Delta,s)$, then $\tau(\alpha)=(-1)^s\sigma(\Delta)$.
\end{Prp}

If $(\Delta,s)$ is a signed SPD with $d$ edges and
$\alpha$ is the floral vertex determined by $(\Delta,s)$, then
we once again abuse notation slightly by defining $\sigma(\alpha):=\sigma(\Delta)$.

\begin{Prp}
The following hold:
(i) If $\alpha$ is the 1-dimensional half-space (i.e. a cardinal ray), then $\sigma(\alpha)=1$.
(ii) $\sigma(\alpha\wedge\beta)=\sigma(\alpha)\sigma(\beta)$ for all floral vertices $\alpha$, $\beta$.
(iii) $\sigma(\alpha)+(-1)^d\sigma(\overline{\alpha})=0$ for every floral vertex $\alpha$.
\end{Prp}

Figure \ref{flowers4d} displays all congruence types of floral vertices in 4 dimensions, together
with the values of $\sigma$ and $\mu$ for each.

\begin{figure}
$$\begin{array}{|c|c|c|c|}
\hline
 & \hbox{Boolean} & \hbox{Volume} & \hbox{Bouquet sign} \\
\Delta & \hbox{expression} & \mu(\Delta) & \sigma(\Delta) \\
\hline
\arxxxx & 1\cap 2\cap 3 \cap 4 & 1 & +1 \\
\hline
\arxxnxx & (1\cup 2)\cap 3\cap 4 & 3 & -1 \\
\hline
\arxxnxnx & ((1\cap 2)\cup 3)\cap 4 & 5 & -1 \\
\hline
\arxxxnx & (1\cup 2\cup 3)\cap 4 & 7 & +1 \\
\hline
\arxxuxxn & (1\cup 2)\cap (3\cup 4) & 9 & +1 \\

\hline
\arxxuxx & (1\cap 2)\cup (3\cap 4) & 7 & -1 \\
\hline
\arxxxnxn & (1\cap 2\cap 3)\cup 4 & 9 & -1 \\
\hline
\arxxnxnxn & ((1\cup 2)\cap 3)\cup 4 & 11 & +1 \\
\hline
\arxxnxxn & (1\cap 2)\cup 3\cup 4 & 13 & +1 \\
\hline
\arxxxxn & 1\cup 2\cup 3 \cup 4 & 15 & -1 \\
\hline
\end{array}$$
\caption{Floral vertices in four dimensions}
\label{flowers4d}
\end{figure}

\section{Generic Orthotopes: Global Theory}

In this section we define:

\begin{Def}
A generic orthotope of dimension $d$ is an orthogonal polytope for which every singular point is a floral vertex
on $d$ half-spaces.
\end{Def}

From our analysis of the facets of a floral vertex, we see:

\begin{Prp}
If $P$ is a generic orthotope, then every face of $P$ of dimension $k$ is a generic 
orthotope of dimension $k$.
\end{Prp}

From our analysis of cross-sections of a floral arrangement, we see:

\begin{Prp}
Suppose $P\subset\mathbb{R}^d$ is a generic orthotope and $\Pi\subset\mathbb{R}^d$ is a
generalized hyperplane.  Then $P\cap\Pi$ is a generic orthotope.
\end{Prp}

We can also say:

\begin{Prp}
The 1-dimensional skeleton of a generic orthotope is a bipartite graph of degree $d$.
\end{Prp}

\begin{proof}
Let $P$ be a generic orthotope and
suppose $\alpha_1$ and $\alpha_2$ are the floral arrangements of adjacent
vertices of $P$.  We will show that $\tau(\alpha_1)+\tau(\alpha_2)=0$.
Assume that the floral vertices $\alpha_1$, $\alpha_2$ are determined by
the signed SPDs $(\Delta_1,s_1)$, $(\Delta_2,s_2)$ respectively.
Assume that the edge joining the vertices is parallel to $e_i$.  Let $(\Delta,s)$
be the SPD that represents the edge cross-section.  
Then, from our analysis of edge cross-sections, we have 
$\Delta=\Delta_1\backslash i=\Delta_2\backslash i$.
Let $s_{1,i},s_{2,i}\in\{\pm 1\}$ be the $i$th component of $s_1$, $s_2$, respectively.
The key observation is that the cardinal direction of the edge starting
at one of the vertices is the negative of the cardinal direction of the edge
starting at the other vertex.
Recall from our discussion of edge orientations that the cardinal ray generated by
$\sigma(\Delta_1)\sigma(\Delta_1\backslash i)s_{1,i}e_i$
is an edge of $\alpha_1$,
while the cardinal ray generated generated by
$\sigma(\Delta_2)\sigma(\Delta_2\backslash i)s_{2,i}e_i$
is an edge of $\alpha_2$.  Thus, whether $s_{1,i}$ and $s_{2,i}$ have
equal or opposite sign, we have
$$\tau(\Delta_1,s_1)=(-1)^{s_1}\sigma(\Delta_1)=-(-1)^{s_2}\sigma(\Delta_2)=-\tau(\Delta_2,s_2).$$
\end{proof}

\subsection{Approximation}

For a pair $P,Q\subset\mathbb{R}^d$ of compact subsets, define
Define the {\it Hausdorff distance function} by
$$\mathrm{Hdist}(P,Q)=\max \left\{\sup_{p\in P} d(p,Q),\sup_{q\in Q} d(q,P)\right\},$$
where
$$d(x,P)=\inf_{p\in P} \|x-p\|_\infty$$
denotes the distance from a point $x$ to $P$ induced by the $L^\infty$ norm.

\begin{Thm}
Given any compact subset $E\subset\mathbb{R}^d$ and any $\epsilon>0$,
there is a generic orthotope $P$ such that $\mathrm{Hdist}(E,P)<\epsilon$.
\label{density}
\end{Thm}

\begin{Lem}
Suppose $P,Q\subset\mathbb{R}^d$ are generic orthotopes that have no supporting hyperplanes
in common.  Then $P\cap Q$ and $P\cup Q$ are generic orthotopes.
\end{Lem}

\begin{proof}
We must show that every vertex of $P\cap Q$ and every vertex of $P\cup Q$ is floral.
Suppose $v$ is a vertex of $P\cap Q$.  Then there are floral arrangements $\alpha$
and $\beta$ such that the tangent cone at $v$ in $P$ (respectively $Q$) is $\alpha$
(respectively $\beta$).  However, $P$ and $Q$ have no supporting hyperplanes in common,
so the floral arrangement at $v\in P\cap Q$ is $\alpha\wedge\beta$.  However, since
$P$ and $Q$ have no common supporting hyperplanes, $\alpha\wedge\beta$ is represented
by a read-once Boolean function.  The same argument holds for the case when
$v$ is a vertex of the union $P\cup Q$.
\end{proof}

Now we argue a proof of Theorem \ref{density}.
Suppose a compact set $E\subset\mathbb{R}^d$ and $\epsilon>0$ are given.  
First choose a rational orthogonal polytope $Q_0$ such that $\mathrm{Hdist}(Q_0,E)<\frac{\epsilon}{2}$.
That $Q_0$ exists is guaranteed by the compactness of $E$.
Let $n$ be the least positive integer such that $Q_1=nQ_0$ is an integral orthotope.
Let $S\subset\mathbb{Z}^d$ be the finite set such that $Q_1$ is a union of translates
$$Q_1=\bigcup_{v\in S} (v+f_v),$$
where $f_v$ is a non-empty face of the standard unit cube $I^d$ for each $v\in S$.
For each $v\in S$, choose an orthogonally aligned $d$-dimensional box $B_v$ such that
(i) $v+f_v$ lies in the interior of $B_v$,
(ii) $\mathrm{Hdist}(B_v,v+f_v)<\frac{\epsilon}{2}$,
and 
(iii) no two of the boxes $B_v$ share a supporting hyperplane.
That these boxes $B_v$ exist is justified by the fact that $S$ is finite.
Finally, let 
$$P=\frac{1}{n}\bigcup_{v\in S} B_v.$$
Evidently we have
$$\mathrm{Hdist}(P,Q_0)<\frac{\epsilon}{n}.$$
From the triangle inequality,
$$\mathrm{Hdist}(P,Q)\leq\mathrm{Hdist}(P,Q_0)+\mathrm{Hdist}(Q_0,Q)<\frac{\epsilon}{2n}+\frac{\epsilon}{2}=
\frac{1+n}{2n}\epsilon\leq \epsilon.$$
Moreover, since the supporting hyperplanes of the boxes $B_v$ for $v\in S$ are distinct,
$P$ is a generic orthotope.  

The lemma above also helps to see why we regard generic orthotopes as ``generic'':
Suppose $P$ is a generic orthotope and $\Pi$ is a supporting hyperplane. 
Then we may ``shift'' $\Pi$ in a direction perpendicular to $\Pi$ while leaving all other
supporting hyperplanes fixed.  For example, one may accomplish such a transformation
using a piece-wise linear function.  One may use such a shift, provided it does not pass
across another supporting hyperplane parallel to $\Pi$, to construct another generic orthotope
$P'$ which has the same face poset as $P$.

The space of generic orthotopes is not open in the space of all orthogonal
polytopes, as the following example demonstrates.  Let $P=[0,2]\times[0,2]\times[0,1]$ and for each $\epsilon>0$
define
$$Q_\epsilon=P\cup\left([0,1]\times[0,1]\times[1,1+\epsilon]\right)\cup\left([1,2]\times[1,2]\times[1,1+\epsilon]\right).$$
Then $P$ is a generic orthotope, but $Q_\epsilon$ is an orthogonal polytope which is never
a generic orthotope.  Moreover, we have
$\lim_{\epsilon\rightarrow 0}\mathrm{Hdist}(P,Q_\epsilon)=0.$

\subsection{Volume and Euler characteristic}

This section shows several combinatorial formulas concerning generic orthotopes.

\begin{Thm}
Suppose $P$ is an integral generic orthotope.  Then
$$\mathrm{volume}(P)=2^{-d}\sum_{v\in P\cap\mathbb{Z}^d} \mu_d(\alpha(v)),$$
where $\alpha(v)$ denotes the floral arrangement at $v$.
\end{Thm}

The formula is easy to
understand:  For any point $v\in P\cap\mathbb{Z}^d$, the fraction $\mu(v)/2^d$ is the
volume of $(v+[-1/2,1/2]^d)\cap P$.  To compute the total volume, simply add all of these values.

If $P$ is an integral generic orthotope, let $n_\alpha$ denote the number of points
in $P\cap\mathbb{Z}^d$ of floral type $\alpha$.  Then we may write the formula above
as
$$\mathrm{volume}(P)=2^{-d}\sum_\alpha \mu_d(\alpha)n_\alpha,$$
where we sum over all congruence types of floral arrangements.

We also have a determinantal expression for the volume of a generic orthotope:

\begin{Thm}
Suppose $P$ is a rational generic orthotope.  For each vertex of $v$,
denote the coordinates where $v=(v_1,v_2,...,v_d)$.
Then
$$\mathrm{volume}(P)=\sum_v \tau(v)\prod_{i=1}^d v_i,$$
summing over all vertices $v\in P$ and $\tau(v)$ denotes the signed
volume of the floral arrangement at $v$.
\label{determinantal}
\end{Thm}

\begin{proof}
First we show that the formula holds for an integral generic orthotope $P$.
Assuming this, we may subdivide $P$ as the union
$$P=\bigcup_{v\in S} (v+I^d),$$
where $S\subset\mathbb{Z}^d$ is finite.
The formula holds for each unit cube $v+I^d$, so we have
$$1=\mathrm{volume}(v+I^d)=\prod_{i=1}^d((v_i+1)-v_i).$$
If we sum these over all of $S$, then we obtain an expression
$$\mathrm{volume}(P)=\sum_{v\in P\cap \mathbb{Z}^d}c_v\prod_{i=1}^d v_i,$$
where $c_v$ denotes a coefficient that depends on $v$.  The key observation is that
the coefficient $c_v$ vanishes exactly when the floral arrangement at $v$ occupies an
even number of orthants.  Moreover, the floral arrangements of $P$ that are occupied by
an odd number of orthants coincide with the (floral) vertices of $P$.
One then verifies that the coefficient $c_v$ is indeed equal to $\tau(v)$ for
every floral vertex.  Since the formula (\ref{determinantal}) holds for every integral generic orthotope, it
also holds for every rational generic orthotope.  
\end{proof}

We have a similar formula for the Euler characteristic of a generic orthotope.
Suppose $P$ is a generic orthotope.  Define 
$$\sigma(P)=\sum_v \sigma(\alpha(v))=\sum_{\alpha}\sigma(\alpha)n_\alpha.$$
where the first sum is over all vertices $v\in P$ the second sum is over floral types $\alpha$.

\begin{Thm}
Suppose $P$ is a generic orthotope with Euler characteristic $\chi(P)$.  Then
$$\chi(P)=2^{-d}\sigma(P).$$
\label{eulercharthm}
\end{Thm}

To establish this, we first prove that $\sigma$ is a valuation when restricted
to generic orthotopes:

\begin{Prp}
If all four of $\{P,Q,P\cap Q,P\cup Q\}$ are generic orthotopes, then
$\sigma(P)+\sigma(Q)=\sigma(P\cap Q)+\sigma(P\cup Q)$.
\label{genericquadruple}
\end{Prp}

We facilitate this by use of a lemma.

\begin{Lem}
Suppose $(\alpha,\beta)$ is a pair of floral arrangements such that
both of $\alpha\cap\beta$ and $\alpha\cup\beta$ are floral arrangements.  
Then $\alpha$ and $\beta$ have no opposite half-spaces.
\end{Lem}

\begin{proof}
In seeking a contradiction, assume that $\alpha$ and
$\beta$ have opposite half-spaces.  Let $f$ and $g$ be facets that have opposite half-planes, 
and let $\Pi$ be the hyperplane containing $f$ and $g$.  Let $f^\circ$ and $g^\circ$
denote the relative interiors of $f$ and $g$ respectively with respect to $\Pi$.
Suppose first that $f^\circ\cap g^\circ$ is empty.  Then the disjunction $\alpha\vee\beta$ has
two genericity regions with opposite outward normal vectors, contradicting the assumption
that $\alpha\cap\beta$ and $\alpha\cup\beta$ are floral arrangements.  
On the other hand, suppose $f^\circ\cap g^\circ$ is non-empty.
Then the join $\alpha\wedge\beta$ is not a pure $d$-dimensional orthotope, so again this contradicts
the assumption that $(\alpha,\beta)$ is a floral pair.  
\end{proof}

Now we prove Proposition \ref{genericquadruple}.
This follows by relating $\sigma$ to the signed volume
function $\tau_d$.
Thus, suppose $\alpha$, $\beta$, $\alpha\cap\beta$, and $\alpha\cup\beta$ are floral
arrangements.  From the lemma,
we may assume without loss of generality that $\alpha$ and $\beta$ are both
represented by signed SPDs, where every edge is marked positive.
Being a sum
of signs of orthants, $\tau_d$ trivially satisfies the inclusion-exclusion rule.
In this case, since all of the half-spaces are positive, this implies that
$\sigma$ also satisfies the inclusion-exclusion rule.  One may verify that
$\sigma(B)=2^d$ for every pure $d$-dimensional axis-aligned box $B\subset\mathbb{R}^d$.
Thus, $\sigma$ is a valuation when restricted to generic orthotopes.
Since $\sigma$ is constant on axis-aligned boxes, it yields a multiple of the
Euler characteristic.

{\bf Example.}  Suppose $d=4$.  Then
the formula in Theorem \ref{eulercharthm} says
\begin{align*}
 & n_{\arxxxx}-n_{\arxxnxx}-n_{\arxxnxnx}+n_{\arxxxnx}+n_{\arxxuxxn} \\
& -n_{\arxxuxx}-n_{\arxxxnxn}+n_{\arxxnxnxn}+n_{\arxxnxxn}-n_{\arxxxxn} \\
& =2^4\cdot\chi(P). 
\end{align*}
We invite the reader to attempt to assemble 4-dimensional generic orthotopes
for experimentation.

\section{Conclusion and open problems}

Having established a theory of generic orthotopes, we pose several
questions.

\subsection{Generic polyconvex polytopes}

Define a {\it polyconvex polytope} as a subset of $\mathbb{R}^d$ that can be
formed as the union of finitely many convex polytopes. 
Define a define a {\it generic polyconvex polytope} as a polyconvex polytope
such that the tangent cone at every vertex is described by applying a read-once Boolean 
function to a set of $d$ half-spaces with distinct supporting hyperplanes.
Thus, in a generic polyconvex polytope, every vertex can be transformed via a linear transformation
to a floral vertex.
Clearly every face of a generic polyconvex polytope is a generic polyconvex polytope.  Is a similar
statement valid for cross-sections?  

\subsection{Discrete Morse theory}

Suppose $P$ is a polyconvex polytope.  As Bieri and Nef describe in
\cite{BN_1983} and \cite{BN_1985}, one may compute the volume
and Euler characteristic of $P$ using ``sweep plane'' algorithms.
Their algorithms compute the volume and Euler characteristic by adding local statistics
as the level sets of a linear functional (essentially a discrete analogue of a Morse function)
pass across the vertices of $P$.  Can we refine theses algorithms
to handle generic orthotopes or generic polyconvex polytopes specifically?
What effect does this have on the complexity of the problem of computing
the volume and Euler characteristic?

\subsection{Genericization}

Let $P$ be an orthogonal polytope.  The problem here is to study methods
of approximating $P$ by a generic orthotope.  How can we accomplish this in a general way?  
An obvious place
to start is to analyze local orthotopal arrangements which are not floral arrangements.
For example, it is not hard to imagine ``perturbing'' one or more of the supporting planes 
in the degenerate vertices appearing in Figure \ref{false3dvertices} to obtain a pair of nearby floral vertices.
More generally, this author imagines ``blow ups'' along
singular (non-floral) faces, obtained by systematically uniting $P$ with generic orthotopes
which ``cover'' the singular faces.  What are the most efficient algorithms for
generizicing a given orthogonal polytope?

\subsection{Simplicial orthotopal arrangements}

Is every simplicial orthotopal arrangement floral?  
We have seen that the face lattice of every floral arrangement coincides with that
of a simplex.  We ask whether or not the converse is also valid.  Thus, given an orthotopal
arrangement $\alpha\subset\mathbb{R}^d$ such that the face poset of $\alpha$ coincides
with a simplex, we wonder whether there is necessarily a read-once Boolean function which
defines $\alpha$.  An affirmative answer would significantly strengthen this author's thesis
that generic orthotopes represent an elementary generalization of the cube to
general orthogonal polytopes.

\subsection{Flag orthotopes}

Suppose $P$ is a $d$-dimensional convex polytope with
face lattice $\mathcal{L}(P)$ and 
$\mathcal{L}{(P)}\xrightarrow{\ \phi\ }\mathbb{R}$.
For each complete
flag $(\emptyset\subset f_0\subset f_1\subset f_2\subset ...\subset f_{d-1}\subset P)$ 
in $\mathcal{L}(P)$ (where $\dim(f_i)=i$), this yields a point 
$$(\phi(f_0),\phi(f_1),\phi(f_2),...,\phi(f_{d-1}))\in\mathbb{R}^d,$$
and the assembly of these points may or may not coincide with the vertices
of a generic orthotope.
Which pairs $(P,\phi)$ does this yield the vertices of a generic orthotope
that respects the face lattice $\mathcal{L}(P)$?  Can we develop efficient algorithms
for realizing $\mathcal{L}(P)$ by a generic orthotope?

Figure \ref{flagorthotope} shows an example of this idea when $d=3$.  
The left part of this figure represents a drawing of a Schlegel diagram of a 
3-dimensional convex polytope, say $P$.  Notice that every vertex, edge, and facet
is marked by a value $\phi(f_i)$.
The right part of this figure illustrates an axonometric
projection of the ``flag orthotope'' given by the data $(P,\phi)$.  
Notice that the polygonal regions on the right
are marked with distances to the three coordinate planes and this yields
a 3-dimensional generic orthotope which ``displays'' the entire face lattice of $P$.

\begin{figure} 
\centering 
\includegraphics[width=\textwidth]{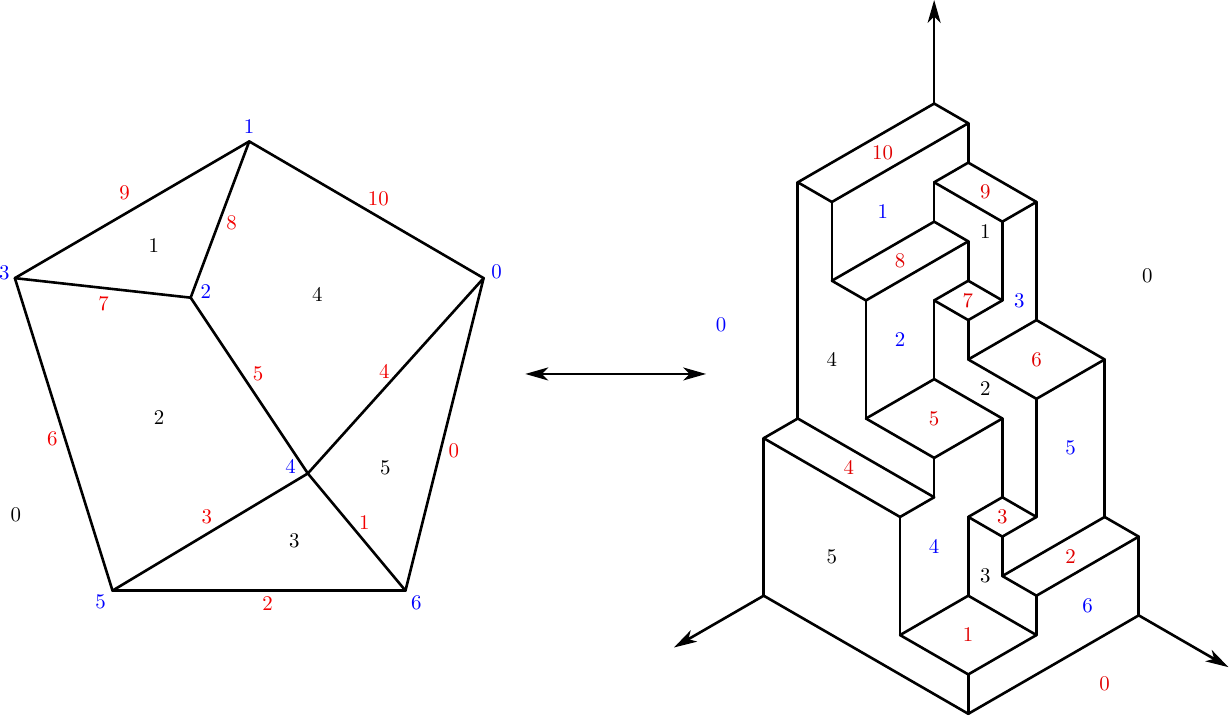}
\caption{A flag orthotope.}
\label{flagorthotope}
\end{figure}

\subsection{Shadows of 4D generic orthotopes}

As this author noticed in \cite{richter_rectangulations}, one may construct 3-dimensional flag
orthotopes
from generic rectangulations.  This idea is easy to conceive due to the extremely limited number
of 3-dimensional floral vertices.  What about the next higher dimension?  Thus, whereas a generic
rectangulation represents a 2-dimensional projection of a flag orthotope of a non-separable planar
map, what are the analogous configurations when considering 3-dimensional projections of a 4-dimensional
generic orthotope?  Due to the number of different types of 4-dimensional floral vertices,
this project appears to be quite large.

\subsection{Coxeter complexes}

The original motivation of this work came from a desire to realize Coxeter complexes
by orthogonal polytopes, and we may now state this problem precisely.
Suppose $(G,S)$ is a Coxeter system of finite type, where 
$S=\{\sigma_1,\sigma_2,...,\sigma_d\}$ is the set of generating involutions of $G$
and $d$ is the rank of $(G,S)$.
Define a {\it generic orthotopal realization} of $(G,S)$ as a generic orthotope 
$P=P_{(G,S)}\subset\mathbb{R}^{d}$
such that there is a bijection $\phi:G\rightarrow \mathcal{L}_0(P)$ (the vertices of $P$)
where two vertices $v_1=\phi(g_1)$, $v_2=\phi(g_2)$ of $P$ are connected by an edge parallel
to the $i$th coordinate axis whenever $g_1=\sigma_ig_2$.

\begin{figure} 
\centering 
\includegraphics[width=0.5\textwidth]{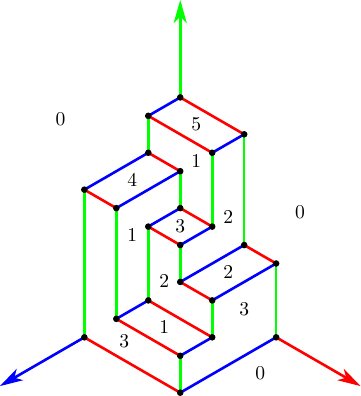}
\caption{The Coxeter complex of type $A_3$.}
\label{axonometric_a3}
\end{figure}

Figure \ref{axonometric_a3} displays an example of this idea.  
The underlying Coxeter
group $G$ is the symmetric group on 4 letters and the generators are $S=\{b,g,r\}$,
subject to the relations
$$\{b^2,g^2,r^2,(bg)^3,(br)^2,(gr)^3\}.$$  
The figure depicts an axonometric projection
of a realization of $(G,S)$ as a generic orthotope.  As in our discussion of flag orthotopes
above, the polygonal regions are marked by distances to coordinate planes in $\mathbb{R}^3$.
Notice in particular that there are 24 vertices, corresponding to the elements of $G$.
One also notices that all of the edges sharing a common color are mutually parallel
and that the 1-dimensional skeleton comprises a Cayley graph of $G$ with the generators $\{b,g,r\}$.
Although this example was first conceived in the context of graph drawing
(as in \cite{eppstein_2008} and \cite{EM_2014} for example), 
we notice that this representation is faithful to the entire
Coxeter complex of the corresponding Coxeter system.  Thus, for every $k\in\{0,1,2,3\}$,
the $k$-dimensional faces of this polytope correspond to
cosets of the Coxeter subgroups generated by $k$ elements of $\{b,g,r\}$.

This author is interested in realizing every finite Coxeter complex by a generic orthotope.
Ideally, one would like to see a uniform system for realizing each of the  three infinite
sequences $A_d$, $BC_d$, $D_d$ of spherical Coxeter systems.  Aside from formulating
it precisely, this author has made little progress
on this problem.  For example, it is possible to realize every finite rank-3 Coxeter system
by a generic orthotope. (This is a good exercise for the interested reader.)  However, the problem is
daunting as soon as $d\geq 4$.
For various reasons, this author suspects that no realization
of the $D_4$ Coxeter complex as a generic orthotope exists.  Since $D_4$ occurs as a subdiagram
of $D_d$ for all $d\geq 5$ and of the exceptional series $E_d$, a negative result
would imply that none of these particular Coxeter systems has a realization as a generic
orthotope.  How should we handle infinite Coxeter systems?

\end{document}

%% file: graphics/dualdiagram.pdf_tex
\begingroup%
  \makeatletter%
  \providecommand\color[2][]{%
    \errmessage{(Inkscape) Color is used for the text in Inkscape, but the package 'color.sty' is not loaded}%
    \renewcommand\color[2][]{}%
  }%
  \providecommand\transparent[1]{%
    \errmessage{(Inkscape) Transparency is used (non-zero) for the text in Inkscape, but the package 'transparent.sty' is not loaded}%
    \renewcommand\transparent[1]{}%
  }%
  \providecommand\rotatebox[2]{#2}%
  \newcommand*\fsize{\dimexpr\f@size pt\relax}%
  \newcommand*\lineheight[1]{\fontsize{\fsize}{#1\fsize}\selectfont}%
  \ifx\svgwidth\undefined%
    \setlength{\unitlength}{118.430424bp}%
    \ifx\svgscale\undefined%
      \relax%
    \else%
      \setlength{\unitlength}{\unitlength * \real{\svgscale}}%
    \fi%
  \else%
    \setlength{\unitlength}{\svgwidth}%
  \fi%
  \global\let\svgwidth\undefined%
  \global\let\svgscale\undefined%
  \makeatother%
  \begin{picture}(1,0.81376896)%
    \lineheight{1}%
    \setlength\tabcolsep{0pt}%
    \put(0,0){\includegraphics[width=\unitlength,page=1]{dualdiagram.pdf}}%
    \put(-0.00279536,0.67382435){\makebox(0,0)[lt]{\lineheight{1.25}\smash{\begin{tabular}[t]{l}$1$\end{tabular}}}}%
    \put(0.14635737,0.5782581){\makebox(0,0)[lt]{\lineheight{1.25}\smash{\begin{tabular}[t]{l}$2$\end{tabular}}}}%
    \put(0.04308497,0.38907151){\makebox(0,0)[lt]{\lineheight{1.25}\smash{\begin{tabular}[t]{l}$3$\end{tabular}}}}%
    \put(0.27961313,0.45514361){\makebox(0,0)[lt]{\lineheight{1.25}\smash{\begin{tabular}[t]{l}$4$\end{tabular}}}}%
    \put(0.16732543,0.14320739){\makebox(0,0)[lt]{\lineheight{1.25}\smash{\begin{tabular}[t]{l}$5$\end{tabular}}}}%
    \put(0,0){\includegraphics[width=\unitlength,page=2]{dualdiagram.pdf}}%
    \put(0.66985197,0.6691382){\makebox(0,0)[lt]{\lineheight{1.25}\smash{\begin{tabular}[t]{l}$\overline{1}$\end{tabular}}}}%
    \put(0.66900576,0.38333681){\makebox(0,0)[lt]{\lineheight{1.25}\smash{\begin{tabular}[t]{l}$\overline{2}$\end{tabular}}}}%
    \put(0.84435139,0.51678523){\makebox(0,0)[lt]{\lineheight{1.25}\smash{\begin{tabular}[t]{l}$\overline{3}$\end{tabular}}}}%
    \put(0.7027003,0.12541858){\makebox(0,0)[lt]{\lineheight{1.25}\smash{\begin{tabular}[t]{l}$\overline{4}$\end{tabular}}}}%
    \put(0.95205283,0.34692649){\makebox(0,0)[lt]{\lineheight{1.25}\smash{\begin{tabular}[t]{l}$\overline{5}$\end{tabular}}}}%
  \end{picture}%
\endgroup%

%% file: graphics/flowers2d.pdf_tex
\begingroup%
  \makeatletter%
  \providecommand\color[2][]{%
    \errmessage{(Inkscape) Color is used for the text in Inkscape, but the package 'color.sty' is not loaded}%
    \renewcommand\color[2][]{}%
  }%
  \providecommand\transparent[1]{%
    \errmessage{(Inkscape) Transparency is used (non-zero) for the text in Inkscape, but the package 'transparent.sty' is not loaded}%
    \renewcommand\transparent[1]{}%
  }%
  \providecommand\rotatebox[2]{#2}%
  \newcommand*\fsize{\dimexpr\f@size pt\relax}%
  \newcommand*\lineheight[1]{\fontsize{\fsize}{#1\fsize}\selectfont}%
  \ifx\svgwidth\undefined%
    \setlength{\unitlength}{274.39273671bp}%
    \ifx\svgscale\undefined%
      \relax%
    \else%
      \setlength{\unitlength}{\unitlength * \real{\svgscale}}%
    \fi%
  \else%
    \setlength{\unitlength}{\svgwidth}%
  \fi%
  \global\let\svgwidth\undefined%
  \global\let\svgscale\undefined%
  \makeatother%
  \begin{picture}(1,0.53572851)%
    \lineheight{1}%
    \setlength\tabcolsep{0pt}%
    \put(0,0){\includegraphics[width=\unitlength,page=1]{flowers2d.pdf}}%
    \put(0.0125922,0.42502944){\color[rgb]{0,0,0}\makebox(0,0)[t]{\lineheight{0}\smash{\begin{tabular}[t]{c}$($\end{tabular}}}}%
    \put(0.32692268,0.42502944){\color[rgb]{0,0,0}\makebox(0,0)[t]{\lineheight{0}\smash{\begin{tabular}[t]{c}$,$\end{tabular}}}}%
    \put(0.51825427,0.42502944){\color[rgb]{0,0,0}\makebox(0,0)[t]{\lineheight{0}\smash{\begin{tabular}[t]{c}$)$\end{tabular}}}}%
    \put(0,0){\includegraphics[width=\unitlength,page=2]{flowers2d.pdf}}%
    \put(0.63509762,0.43215366){\makebox(0,0)[lt]{\lineheight{1.25}\smash{\begin{tabular}[t]{l}$1$\end{tabular}}}}%
    \put(0.69231575,0.43194145){\makebox(0,0)[lt]{\lineheight{1.25}\smash{\begin{tabular}[t]{l}$2$\end{tabular}}}}%
    \put(0.0125922,0.11069896){\color[rgb]{0,0,0}\makebox(0,0)[t]{\lineheight{0}\smash{\begin{tabular}[t]{c}$($\end{tabular}}}}%
    \put(0.32692268,0.11069896){\color[rgb]{0,0,0}\makebox(0,0)[t]{\lineheight{0}\smash{\begin{tabular}[t]{c}$,$\end{tabular}}}}%
    \put(0.51825427,0.11069896){\color[rgb]{0,0,0}\makebox(0,0)[t]{\lineheight{0}\smash{\begin{tabular}[t]{c}$)$\end{tabular}}}}%
    \put(0,0){\includegraphics[width=\unitlength,page=3]{flowers2d.pdf}}%
    \put(0.65305623,0.13296665){\makebox(0,0)[lt]{\lineheight{1.25}\smash{\begin{tabular}[t]{l}$1$\end{tabular}}}}%
    \put(0.65305623,0.06262883){\makebox(0,0)[lt]{\lineheight{1.25}\smash{\begin{tabular}[t]{l}$2$\end{tabular}}}}%
  \end{picture}%
\endgroup%

%% file: genericorthotopes_arxiv.bbl
\begin{thebibliography}{99}

\bibitem{AA_2001}
Antonio Aguilera and Dolors Ayala.
Converting orthogonal polyhedra from extreme vertices model to 
B-Rep and to alternating sum of volumes.  In: {\it Geometric modelling (Dagstuhl, 1999)}, 1--18,
Comput. Suppl., 14, Springer, Vienna, 2001.

\bibitem{AB_2006}
Gadi Aleksandrowicz, and Gill Barequet.
Counting $d$-dimensional polycubes and nonrectangular planar polyominoes. 
In:  {\it Computing and combinatorics,}  418--427,
Lecture Notes in Comput. Sci., 4112, Springer, Berlin, 2006. 

\bibitem{BBR_2010}
Ronnie Barequet, Gill Barequet, and G\"unter Rote.
Formulae and growth rates of high-dimensional polycubes.
{\it Combinatorica} {\bf 30} (2010), no. 3, 257--275. 

\bibitem{BDPR_2007}
Antonio Bernini, Filippo Disanto, Renzo Pinzani, and Simone Rinaldi.
Permutations defining convex permutominoes.
{\it J. Integer Seq.} {\bf 10} (2007), no. 9, Article 07.9.7, 26 pp. 

\bibitem{BC_2008}
K\'{a}roly Bezdek and Robert Connelly.
On the weighted Kneser-Poulsen conjecture.
{\it Period.\ Math.\ Hungar.} 57 (2008), no. 2, 121--129.


\bibitem{BN_1983}
Hanspeter Bieri and Walter Nef.
A sweep-plane algorithm for computing the volume of polyhedra represented in Boolean form.
{\it Linear Algebra Appl.} {\bf 52/53} (1983), 69--97.

\bibitem{BN_1985}
Hanspeter Bieri and Walter Nef.
A sweep-plane algorithm for computing the Euler-characteristic of polyhedra represented 
in Boolean form.
{\it Computing} {\bf 34} (1985), no. 4, 287--302.


\bibitem{BMP_1999}
Olivier Bournez, Oded Maler, and Amir Pnueli. 
Orthogonal polyhedra: Representation and computation. 
In: {\it Hybrid Systems: Computation and Control, 
Nijmegen, The Netherlands, 29-31 March 1999.} 
Lecture Notes in Computer Science, vol 1569, Springer, 2000,
46--60.

\bibitem{breen_1996}
Marilyn Breen.
Staircase kernels for orthogonal $d$-polytopes. 
{\it Monatsh. Math.} {\bf 122} (1996), no. 1, 1--7.


\bibitem{coxeter_1963}
H. S. M. Coxeter.
{\it Regular Polytopes.}
2nd ed.
Macmillan, 1963.

\bibitem{CH_2011}
Yves Crama and Peter L.\ Hammer.
{\it Boolean functions. Theory, algorithms, and applications.}
Encyclopedia of Mathematics and its Applications, vol.\ 142. 
Cambridge University Press, Cambridge, 2011. 

\bibitem{csikos_2001}
Bal\'{a}zs Csik\'{o}s.
On the volume of flowers in space forms.
{\it Geom.\ Dedicata} {\bf 86} (2001), no. 1-3, 59--79. 


\bibitem{DG_1982}
Ludwig Danzer and Branko Gr\"unbaum.
Intersection properties of boxes in $\mathbb{R}^d$.
{\it Combinatorica.} {\bf 2} (1982), no. 3, 237--246. 

\bibitem{DV_2005}
Alain Daurat, Maurice Nivat.
Salient and reentrant points of discrete sets.
{\it Discrete Appl. Math.} {\bf 151} (2005), 106--121.

\bibitem{eppstein_2008}
David Eppstein. 
The topology of bendless three-dimensional orthogonal graph drawing. 
{\it Proc. 16th Int. Symp. Graph Drawing (GD 2008)}, pp. 78--89. Springer-Verlag, 
Lecture Notes in Computer Science 5417, 2008.

\bibitem{EM_2014}
David Eppstein and Elena Mumford.
Steinitz theorems for simple orthogonal polyhedra.
{\it J. Comput. Geom.} {\bf 5} (2014), no. 1, 179--244. 

\bibitem{GM_1995}
Yehoram Gordon and Mathieu Meyer.
On the volume of unions and intersections of balls in Euclidean space. 
Geometric aspects of functional analysis (Israel, 1992--1994), 91--101,
{\it Oper.\ Theory Adv.\ Appl.} {\bf 77}, Birkh\"auser, Basel, 1995.

\bibitem{guttmann_2009}
Anthony J.\ Guttmann (editor).
{Polygons, Polyominoes and Polycubes.}
Lecture Notes in Physics, no.\ 775.
Springer Science + Business Media B.V., Netherlands, 2009.

\bibitem{MS_2005}
Ezra Miller and Bernd Sturmfels.
{\it Combinatorial Commutative Algebra.}
Springer, Berlin, 2005.


\bibitem{PA_2006}
Ricardo P\'erez-Aguila.
Orthogonal Polytopes: Study and Application.  (Doctoral dissertation.)
Ciencias de la Computaci\'on. Departamento de Computaci\'on, Electr\'onica, 
Física e Innovaci\'on, Escuela de Ingeniería y Ciencias, 
Universidad de las Am\'ericas Puebla, 2006.

\bibitem{PA_2010}
Ricardo P\'erez-Aguila.
Computing the discrete compactness of orthogonal pseudo-polytopes via their 
$n$D-EVM representation.
{\it Math.\ Probl.\ Eng.} 2010, Art. ID 598910, 28 pp.

\bibitem{PA_2011}
Ricardo P\'erez-Aguila.
Efficient boundary extraction from orthogonal pseudo-polytopes: 
an approach based on the $n$D-EVM.
{\it J.\ Appl.\ Math.} 2011, Art. ID 937263, 29 pp. 

\bibitem{PAAR_2008}
Ricardo P\'{e}rez-Aguila, Antonio Aguilera, and Guillermo Romero.
The odd edge characterization as a combinatorial property of the 
$n$-dimensional orthogonal pseudo-polytopes. 
In: Papers of the Mexican Mathematical Society (Spanish), 23--52,
{\it Aportaciones Mat.\ Comun.}, {\bf 38}, Soc.\ Mat.\ Mexicana, México, 2008. 

\bibitem{richter_rectangulations}
David Richter.
Some notes on generic rectangulations.
To appear, {\it Contrib.\ Disc.\ Math.}


\bibitem{WW_2016}
Michael Werman and Matthew L.\ Wright.
Intrinsic volumes of random cubical complexes.
{\it Discrete Comput.\ Geom.} {\bf 56} (2016), no. 1, 93--113.

\end{thebibliography}
